\documentclass[12pt]{amsart}
\usepackage{amssymb, amstext, amscd, amsmath}
\usepackage{amsthm}
\usepackage{graphicx}
\usepackage{txfonts}
\newtheorem{thm}{Theorem}[section]
\newtheorem{cor}[thm]{Corollary}

\newtheorem{lem}[thm]{Lemma}

\theoremstyle{definition}
\newtheorem{rem}[thm]{Remark}

\newtheorem{defn}[thm]{Definition}

\numberwithin{equation}{section}

\begin{document}

\title[Critical Heegaard surfaces obtained by self-amalgamation]
{Critical Heegaard surfaces obtained by self-amalgamation}

\author{Qiang E}
\address{School of Mathematical Sciences, Dalian University of
Technology, Dalian 116024, China} \email{eqiangdut@gmail.com}

\thanks{This research is supported by a key grant (No.10931005) of
NSFC and a grant (No.11028104) of NSFC }

\author{Fengchun Lei}
\address{School of Mathematical Sciences, Dalian University of
Technology, Dalian 116024, China} \email{ffcclei@yahoo.com.cn}



\begin{abstract}Critical surfaces can be regarded as topological index 2
minimal surfaces which was introduced by David Bachman. In this
paper we give a sufficiently condition and a necessary condition
for self-amalgamated Heegaard surfaces to be critical.
\end{abstract}
\maketitle

\section{Introduction}
Let $F$ be a properly embedded, separating surface with no torus
components in a compact, orientable, irreducible 3-manifold $M$,
dividing $M$ into two submanifolds. Then the $\emph{disk complex}$,
$\Gamma(F)$, is defined as follows:

(1) Vertices of $\Gamma(F)$ are isotopy classes of compressing disks
for $F$.

(2) A set of $m+1$ vertices forms an $m-$simplex if there are
representatives for each that are pairwise disjoint.

David Bachman explored the information which is contained in the
topology of $\Gamma(F)$ by defining the $\emph{topological index}$
of $F$ \cite{3}. If $\Gamma(F)$ is non-empty then the topological
index of $F$ is the smallest $n$ such that $\pi_{n-1}(\Gamma(F))$ is
non-trivial. If $\Gamma(F)$ is empty then $F$ will have topological
index $0$. If $F$ has a well-defined topological index (i.e.
$\Gamma(F)=\emptyset$ or non-contractible) then we will say that $F$
is a \emph{topologically minimal surface}.

By definition, $F$ has topological index 0 if and only if it is
incompressible, and has topological index 1 if and only if it is
strongly irreducible. Critical surfaces, which are also defined by
David Bachman\cite{1}\cite{4},  can be regarded as topological index
2 minimal surfaces\cite{4}.

\begin{defn}\cite{4}
  $F$ is \emph{critical} if the compressing disks for $F$ can be partitioned into two sets $C_{0}$ and $C_{1}$, such that

  (1) for each $i=0,1$, there is at least one pair of disks $V_{i}, W_{i}\in C_{i}$ on opposite sides of $F$ such that $V_{i}\cap W_{i}=\emptyset$;

  (2) if $V\in C_{0}$ and $W \in C_{1}$ are on opposite sides of $F$ then $V\cap W\neq\emptyset$.
\end{defn}

Some critical Heegaard surfaces have been constructed by Jung Hoon
Lee.

\begin{thm}\emph{\cite{9}}
 The standard minimal genus Heegaard splitting of (closed orientable surface)$\times S^{1}$ is a critical Heegaard splitting.
 \end{thm}

Jung Hoon Lee also showed that some critical Heegaard surfaces can
be obtained by amalgamating two strongly irreducible Heegaard
splittings.
\begin{thm}\emph{\cite{9}}
Let $X\cup _{S}Y$ be an amalgamation of two strongly irreducible
Heegaard splittings $V_{1}\cup _{S_{1}}W_{1}$ and $V_{2}\cup
_{S_{2}}W_{2}$ along homeomorphic boundary components of
$\partial_{-}V_{1}$ and $\partial_{-}V_{2}$. Assume that $V_{2}$ is
constructed from $\partial_{-}V_{2}\times I$ by attaching only one
1-handle. If there exist essential disks $D_{1}\subset W_{1}$ and
$D_{2}\subset W_{2}$ which persist into disjoint essential disks in
$Y$ and $X$ respectively, then $S$ is critical.
\end{thm}

The following theorem is the main result of this paper, which states
that some critical Heegaard surfaces can be obtained by
self-amalgamating strongly irreducible Heegaard splittings. Terms in
the theorems will be defined in Section 2.
\begin{thm}\
 Suppose $M$ is an irreducible 3-manifold with two homeomorphic boundary
  components $F_{1}$ and $F_{2}$, and $V\cup _{S}W$ is a strongly irreducible Heegaard splitting of M such that $F_{1}\cup F_{2}\subset
\partial_{-}W$.
 Suppose $M$ admits an essential disk $B$ in $V$ and two
spanning annuli $A_{1}$, $A_{2}$ in $W$, such that $\partial B$,
$\partial_{1} A_{1}$, $\partial_{1} A_{2}$ are disjoint curves on
$S$, and $ \partial_{2}A_{i}\subset F_{i},$ for $i=1,2$. Let $M^*$=
$V^*\cup_{S^*}W^*$  be the self-amalgamation of $M=V\cup _{S}W$,
  such that $\partial_{2}A_{1}$ is identified with $\partial_{2}A_{2}$. Then $S^*$ is a critical Heegaard surface of $M^*$.
\end{thm}

As a corollary, we show a generalized result of Theorem 1.2.
\begin{cor}
Let F be a closed, connected, orientable surface, and let $\varphi
:F \rightarrow F$ be a surface diffeomophism which preserves
orientation. If $d(\varphi)\leq 2$, then the standard Heegaard
surface of the surface bundle $M(F,\varphi)$ is critical.
\end{cor}

We also show a necessary condition for self-amalgamated Heegaard
surfaces to be critical.

\begin{thm}\
 Suppose that
$M^{*}=V^{*}\cup_{S^{*}}W^{*}$ is the self-amalgamation of
$M=V\cup_{S}W$. If $S^{*}$ is a critical Heegaard surface of
$M^{*}$, then $d(S)\leq 2$.
\end{thm}

\section{Preliminaries}
An essential annulus $A$ properly embedded in a compression body $C$
is called a \emph{spanning annulus} if one component of $\partial A$
denoted by $\partial_{1}A$ lies in $\partial_{+}C$, while the other
denoted by $\partial_{2}A$ lies in $\partial_{-}C$.

Let $\emph{M}$ be a compact orientable 3-manifold. If there is a
closed surface $\emph{S}$ which cuts $\emph{M}$ into two compression
bodies $\emph{V}$ and $\emph{W}$ with
$S=\partial_{+}V=\partial_{+}W$, then we say $\emph{M}$ has a
\emph{Heegaard splitting}, denoted by $M=V\cup_{S}W$; and $\emph{S}$
is called a \emph{Heegaard surface} of $\emph{M}$.

A Heegaard splitting $M=V\cup_{S}$W is said to be \emph{reducible}
if there are two essential disks $D_{1}\subset V$ and $D_{2}\subset
W$ such that $\partial D_{1}=\partial D_{2}$; otherwise, it is
\emph{irreducible}. A Heegaard splitting $M=V\cup_{S}$W is said to
be \emph{weakly reducible} if there are two essential disks
$D_{1}\subset V$ and $D_{2}\subset W$ such that $\partial
D_{1}\cap\partial D_{2}=\emptyset$; otherwise , it is \emph{strongly
irreducible}.

The \emph{distance} between two essential simple closed curves
$\alpha$ and $\beta$ in $S$, denoted by $d(\alpha, \beta)$, is the
smallest integer $n \geq 0$ such that there is a sequence of
essential simple closed curves $\alpha=\alpha_{0},
\alpha_{1}...,\alpha_{n}=\beta$ in $S$ such that $\alpha_{i-1}$ is
disjoint from $\alpha_{i}$ for $1\leq i\leq n$.

The \emph{distance} of the Heegaard splitting $V\cup_{S}W$ is
$d(S)=Min \{ d(\alpha, \beta) \},$ where $\alpha$ bounds an
essential disk in $V$ and $\beta$ bounds an essential disk in $W$.
$d(S)$ was first defined by Hempel, see \cite{7}.

Let $M$ be a compact orientable 3-manifold with homeomorphic
boundary components $F_{1}$ and $F_{2}$, and $M=V\cup_{S}W$ be a
Heegaard splitting such that $F_{1}\cup F_{2}\subset
\partial_{-}W$. Let $M^{*}$ be the manifold obtained from $M$ by
gluing $F_1$ and $F_2$ via a homeomorphism $f: F_1\rightarrow F_2$.
Then $M^{*}$ has a natural Heegaard splitting $M^{*}=V^{*}
\cup_{S^{*}}W^{*}$ called the \emph{self-amalgamation} of
$M=V\cup_{S}W$  as follows:

Let $p_{i}$ be a point on $F_{i}$ such that $f(p_1)=p_2$. Note that
$W$ is obtained by attaching 1-handles $h_1,...,h_m$ to
$\partial_{-} W\times I$. Let $\alpha_{i}=p_{i}\times I$,
$\alpha_{i}\times D$ be the regular neighborhood of $\alpha_{i}$ for
$i=1,2$. We may assume that $\alpha_{i}\times D$ is disjoint from
the 1-handles $h_1,...,h_m$, and $f(p_{1}\times D)=p_{2}\times D$.

Now, in the closure of $M^{*}-V$, the arc
$\alpha=\alpha_1\cup\alpha_2$ has a regular neighborhood
$\alpha\times D$ which intersects $\partial_{+}V=S$ in two disks
$D_1$ and $D_2$. We denote by $p$ the point $p_i$, $D$ the disk
$p\times D\subset \alpha\times D$, and $F$ the surface $F_i$ in
$M^*$. Let $V^{*}=V\cup \alpha\times D$ and $W^*$ be the closure of
$M^{*}-V^*$. $V^*$ and $W^*$ are compression bodies. Let $S^{*}$ be
$V^{*}\cap W^{*}$, then $M^{*}=V^{*} \cup_{S^{*}}W^{*}$ is a
Heegaard splitting called the self-amalgamation of $V\cup_{S}W$. It
is clear that $g(S^{*})=g(S)+1$ (Fig.1).

\begin{figure}
\centering
    \includegraphics[width=9cm]{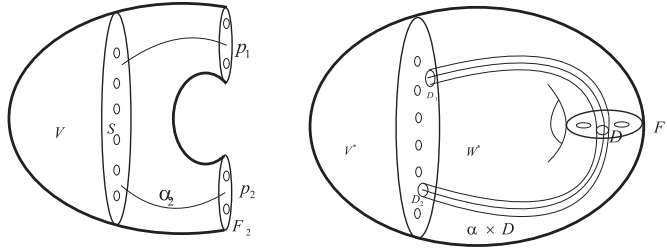}\\
  \caption{$V\cup_{S}W$ and $V^*\cup_{S^*}W^*$}\label{sf}
\end{figure}

\begin{lem}\emph{\cite{11}}
$F-int D$ is incompressible in $W^*$.
\end{lem}

Let $S_1$ be the surface $S-int D_1\cup int D_2$. Then $S_1$ is a
sub-surface of $S$ with two boundary components $\partial D_1$ and
$\partial D_2$.  An essential arc $\gamma$ in $S_1$ is called
\emph{strongly essential} if both two boundary points lie in
$\partial D_i$ and $\gamma$ is an essential arc on $S_{1}\cup D_j$,
where
 \{i, j\}=\{1, 2\}.

\begin{lem}\emph{\cite{11}}
Suppose that $E$ is an essential disk in $V^*$ or $W^*$ and
$\partial E\cap\partial D\neq\emptyset$. Then there exist an arc
$\gamma\in\partial E\cap S_1$ such that $\gamma$ is strongly
essential in $S_1$.
\end{lem}

\begin{lem}\emph{\cite{11}}
Suppose that $E$ is an essential disk in $V^*$ and $|E\cap D|$ is
minimal up to isotopy. Let $\Delta$ be any outermost disk of $E$ cut
by $E\cap D$. Then $\partial \Delta\cap S_1$ is strongly essential
in $S_1$.
\end{lem}

A \emph{surface bundle}, denoted by $M(F,\varphi)$, is a 3-manifold
obtained from $F \times [0, 1]$ by gluing its boundary components
via a surface diffeomorphism $\varphi: F \times \{0\}\rightarrow
F\times \{1 \}$. When $\varphi$ is the identity, $M(F,\varphi)\cong
F\times S^{1}$.

Let $F$ be a closed orientable surface with genus $g(F)\geq 2$.
Suppose that $\varphi$ is a homeomorphism of $F$. The
\emph{translation distance} of $\varphi$ is $d(\varphi)=min\{
d(\alpha, \varphi(\alpha))\} $, where $\alpha$ is an essential
simple closed curve on $F$. $d(\varphi)$ was first defined by
Bachman and Schleimer \cite{5}.

\section{Proofs of Theorem 1.4 and Corollary 1.5}
Now we give the proof of Theorem 1.4. It shows a sufficient
condition for a self-amalgamated Heegaard surface to be critical.
\begin{proof}(of Theorem 1.4.)
Since $V\cup_{S} W$ is strongly irreducible, it follows from Casson
and Gordon's theorem \cite{6} that \emph{F }is incompressible.
 Since $\partial_{2}A_{1}=\partial_{2}A_{2}$, it follows that $A_{1}\cup A_{2}$ is an essential
 annulus in $M^{*}-V$, denoted by \emph{A}. Take a spanning arc $\alpha$ in
 \emph{A}, and let  $V^{*}=V\cup \alpha\times D$ and $W^*$ be the closure
of $M^{*}-V^*$.
 Then $M^{*}=V^{*}\cup_{S^*}W^*$ is obtained by self-amalgamation of $M=V\cup_{S} W$. Now we prove that $S^*$ is a critical Heegaard surface
 of $M^*$.

Let $D$ be a compressing disk of $V^*$ corresponding to the 1-handle
$\alpha\times D$. We give a partition of the compressing disks for
$S^*$, $C_{0}\cup C_{1}$, as follows: (For the sake of convenience,
in the following statement,  `` a disk in $V^*\cap C_{i}$" means ``
a compressing disk in $V^*$ which belongs to $C_i$".)

$V^*\cap C_{0}$ consists of compressing disks in $V^*$ that could be
be isotoped into $V$ but inessential in $V$;

$W^*\cap C_{0}$ consists of compressing disks in $W^*$ that are
disjoint from $D$;

$V^*\cap C_{1}$ consists of compressing disks in $V^*$  that do not
belong to $V^*\cap C_{0}$;

$W^*\cap C_{1}$ consists of compressing disks in $W^*$  that are not
disjoint from $D$.

Each compressing disk of $S^*$ must be contained in $C_{0}$ or
$C_{1}$. Now we need to show $C_{0}\cup C_{1}$ satisfies the
definition of criticality.

     \textbf{ Claim 1.} $C_{0}$ contains a disjoint pair of disks on opposite sides of $S^*$.

 Note that $D$ belongs to $C_0$. Since $\partial_{-}W$ has two components, there exists at least one
essential disk in $W$ disjoint from $\alpha\times D$ . Hence there
exists at least one essential disk in $W^*$ which is disjoint from
$D$. This means $C_{0}$ contains disjoint compressing disks for
$S^*$ on opposite sides.

 \textbf{ Claim 2.} $C_{1}$ contains a disjoint pair of disks on opposite sides of $S^*$.

Since $\alpha$ is contained in the annulus $A$, $cl(A-(\alpha\times
D))$ is an essential disk in $W^*$ intersecting $D$ in at least two
points, so it belongs to $C_{1}$. The essential disk $B$ in $V$
persists as an essential disk in $V^*$ and belongs to $C_{1}$. By
assumption, $cl(A-\alpha\times D)$ is disjoint with $B$. This means
$C_{1}$ also contains disjoint compressing disks for $S^*$ on
opposite sides.

\textbf{ Claim 3.} Any disk in $V^*\cap C_{0}$ intersects any disk
in $W^*\cap C_{1}$.

 Let $E$ be any disk in $W^*$ that intersect with $D$. Let $D_s$ be any disk essential in $V^*$, but inessential in $V$.
Recall $D_1\cup D_2=(\alpha\times D)\cap S$.  If $\partial D_s$ is
isotopic to one of $\partial D_1$ and $\partial D_2$, then $D_s
\cong D$ and there is nothing to prove. So we suppose that $\partial
D_s$ bounds a pair of pants together with $\partial D_1$ and
$\partial D_2$. By Lemma 2.2, there is an arc $\gamma\in\partial
E\cap S_1$ such that $\gamma$ is strongly essential in $S_1$. Note
that a strongly essential arc in $S_1$ must intersect with $\partial
D_s$. Hence $E\cap D_s\neq\emptyset$.

\textbf{ Claim 4.} Any disk in $W^*\cap C_{0}$ intersects any disk
in $V^*\cap C_{1}$.

 Let $E_{0}$ be an essential disk in $W^*$
that is disjoint from $D$. After isotopy, $\partial E_{0}$ can be
made disjoint from $\alpha \times D$. By Lemma 2.1 $E_{0}$ and
$F-int D$ can be made disjoint by a standard innermost disk
argument. This means that $E_{0}$ can be regarded as an essential
disk in $W$. Let $D^{1}$ be an essential disk in $V^*$ that belongs
to $C_1$. For proving Claim 4, we need to show $D^{1}\cap
E_{0}\neq\emptyset$.

 Suppose to the contrary that $D^{1}\cap E_{0}=\emptyset$. We assume that $D^{1}$
is chosen so that the number of components of intersection $|D\cap
D^{1}|$ is minimal up to isotopy of $D^{1}$, satisfying $E_{0}\cap
D^{1}=\emptyset$. First, we suppose $|D\cap D^1|=\emptyset$. Then
$D^1$ can be regard as an essential disk in $V$. Then $D_1\cap
E_0\neq \emptyset$ since $V\cup_{S}W$ is strongly irreducible, a
contradiction.  Hence $|D\cap D^1|\neq\emptyset$. By a standard
innermost disk argument, we can
 assume $D\cap D^{1}$ consists of arc components. Let $\beta\subset S^*$ be an outermost arc component in $D^{1}$ and
 $\Delta$ be the corresponding outermost disk in $D^{1}$. Since the disk $D$ cut $V^*$ into $V$, after a small isotopy $\Delta$ lies in \emph{V}.

By Lemma 2.3, $\beta\cap S_1$ is strongly essential in $S_1$, hence
$\Delta$ is essential in \emph{V}. Since $V\cup_{S}W$ is strongly
irreducible, $\partial\Delta\cap\partial E_{0}\neq\emptyset$. It is
easy to see that $\partial E_{0}\cap\partial\Delta=\partial
E_{0}\cap\beta\subset\partial E_{0}\cap\partial D^{1}$. However, we
have assumed $E_{0}\cap D^{1}=\emptyset$, a contradiction. Claim 4
follows.

Hence  $C_{0}\cup C_{1}$ satisfies the definition of criticality.
This completes the proof of Theorem 1.4.
\end{proof}

\begin{proof}(of Corollary 1.5.)
 The standard Heegaard splitting
of a surface bundle\cite{5} is the self-amalgamation of the type 2
Heegaard spitting of \{closed surface\} $\times I$ \cite{10}. The
genus of the Heegaard surface is $2g(F)+1$. If $d(\varphi)\leq 2$,
it is easy to see $M(F,\varphi)$ satisfies the condition of Theorem
1.4.
\end{proof}

 \begin{rem}\
 There are surface bundles of arbitrarily high genus which have genus two
Heegaard splittings\cite{8}. If $M(F,\varphi)$ contains a strongly
irreducible Heegaard surface $H$, then
$d(\varphi)\leq-\chi(H)$\cite{5}. It follows that if $M(F,\varphi)$
contains an irreducible genus two Heegaard surface, then it also
contains a critical Heegaard surface.
\end{rem}
\section{A necessary condition for self-amalgamated Heegaard surfaces to be critical}
The following result could be found in the proof for the main
theorem in\cite{11}. Recall we suppose
$M^{*}=V^{*}\cup_{S^{*}}W^{*}$ is the self-amalgamation of
$M=V\cup_{S}W$ and $D$ is a meridian disk of $V^*$ corresponding to
the 1-handle attached to $V$.

\begin{lem}\emph{\cite{11}} If $d(S)\geq 3$, for each pair of disks $D^{*}\subset V^*$ and
$E^*\subset W^*$ such that $D^*$ is not isotopic to $D$ and
$\partial E^{*}\cap \partial D\neq\emptyset$, we have $|D^{*}\cap
E^{*}|\geq 2$.
\end{lem}

\begin{proof}
(of Theorem 1.6.) Since $S^*$ is critical, the compressing disks for
$S^*$ can be partitioned into two sets $C_{0}$ and $C_{1}$
satisfying the definition of criticality.

Assume that $D\subset V^{*}\cap C_{0}$.  Each disk in $V^{*}\cap
C_{1}$ is not isotopic to $D$ and each disk in $W^{*}\cap C_{1}$
intersects with $D$. Since $S^*$ is critical, there exists at least
one disjoint pair of disks $D^*\subset V^{*}\cap C_{1}$ and
$E^*\subset W^{*}\cap C_{1}$. By Lemma 4.1, we have $d(S)\leq 2$.
\end{proof}

\end{document}